\documentclass[a4paper,dvipsnames,oneside]{amsart}

\usepackage[left=3cm,right=3cm,top=25mm,bottom=25mm]{geometry}

\usepackage{amsmath,amssymb,amsfonts,amsthm,mathtools,comment,enumitem,thmtools,setspace,graphicx}
\usepackage{tikz}
\usetikzlibrary{patterns}
\usepackage[color,matrix,arrow]{xy}
\definecolor{bluUniud}{RGB}{119,154,171}

\definecolor{mycolor}{HTML}{F7F8E0}
\usepackage[colorlinks,citecolor=mycolor!70!black,urlcolor=black]{hyperref} %cyan

\usepackage[capitalise,noabbrev]{cleveref}

\newtheorem{theorem}{Theorem}[section]
\newtheorem{mtheorem}{Theorem}

\newtheorem{lemma}[theorem]{Lemma}
\newtheorem{property}[theorem]{Property}
\newtheorem{proposition}[theorem]{Proposition}
\newtheorem{construction}[theorem]{Construction}

\theoremstyle{definition}
\newtheorem{definition}[theorem]{Definition}

\theoremstyle{remark}  %possibilit\`{a} di definire a piacere i parametri

 \newtheoremstyle{fatto}% name of the style to be used
  {}% measure of space to leave above the theorem. E.g.: 3pt
  {}% measure of space to leave below the theorem. E.g.: 3pt   \topsep
  {\itshape}% name of font to use in the body of the theorem
  {5mm}% measure of space to indent
  {\bfseries} %name of head font
  {.}% punctuation between head and body
  {2mm}% space after theorem head; " " = normal interword space
  {}% Manually specify head
\theoremstyle{fatto}

\declaretheoremstyle[
  spaceabove=-3pt,%
  spacebelow=6pt,%
  headfont=\normalfont\itshape,%
  postheadspace=2mm,%
  qed=\footnotesize{\(\square\)}]{mystyle}

\setlist{itemsep=.02em,labelsep=.4em,topsep=.5em}   % Modifica spaziatura tra item e distanza tra etichetta e contenuto item

\newcommand{\rca}{\mathsf{RCA}_0}
\newcommand{\wkl}{\mathsf{WKL}_0}
\newcommand{\aca}{\mathsf{ACA}_0}

\newcommand{\E}{\,E\,}
\newcommand{\Q}{\,Q\,}

\newcommand{\Qcon}{\,\bar{Q}\,}
\newcommand{\minpath}{minimal path}

\newcommand{\imp}{\rightarrow}
\newcommand{\Imp}{\Rightarrow}

\newcommand{\nat}{\mathbb{N}}
\newcommand{\Z}{\mathsf{Z}}

\newcommand{\smf}{{}^\smallfrown}

\makeatletter
\@namedef{subjclassname@2020}{%
  \textup{2020} Mathematics Subject Classification}
\makeatother

\definecolor{azzurro}{HTML}{41C6C4}
\definecolor{viola}{HTML}{D628D3}
\definecolor{oliva}{HTML}{DC80BA}

\title[Uniquely orderable interval graphs]{Uniquely orderable interval graphs}
\author[M.~Fiori-Carones]{Marta Fiori-Carones}
\address{Instytut Matematyki, Uniwersytet  Warszawski,
Banacha 2, 02-097 Warszawa --- Poland}
\email{marta.fioricarones@outlook.it}
\author[A.~Marcone]{Alberto Marcone}
\address{Dipartimento di scienze matematiche, informatiche e fisiche, Universit\`a di Udine, Via delle Scienze 208, 33100 Udine --- Italy}
\email{alberto.marcone@uniud.it}
\thanks{Both authors were partially supported by the Italian PRIN 2017 Grant \lq\lq Mathematical
Logic: models, sets, computability\rq\rq.}
\date{\today}

\begin{document}

\begin{abstract}
Interval graphs and interval orders are deeply linked. In fact, edges of an
interval graphs represent the incomparability relation of an interval
order, and in general, of different interval orders. The question about the
conditions under which a given interval graph is associated to a unique
interval order (up to duality) arises naturally. Fishburn provided a
characterisation for uniquely orderable finite connected interval graphs.
We show, by an entirely new proof, that the same characterisation holds
also for infinite connected interval graphs. Using tools from reverse
mathematics, we explain why the characterisation cannot be lifted from the
finite to the infinite by compactness, as it often happens.
\end{abstract}

\keywords{Interval graphs, infinite graphs, unique orderability, reverse
mathematics}

\subjclass[2020]{Primary 05C63; Secondary 05C75, 03B30, 05C62}

\maketitle

\section{Introduction}

An \emph{interval graph} is a graph whose vertices can be mapped (by an
\emph{interval representation}) to nonempty intervals of a linear order in
such a way that two vertices are adjacent if and only if the intervals
associated to them intersect (it is thus convenient to assume that the
adjacency relation is reflexive). Consequently, if two vertices are
incomparable in the graph, the corresponding intervals are placed one before
the other in the linear order. The definition of interval graphs leads to an
analogous concept for partial orders. In fact, a partial order $<_P$ is an
\emph{interval order} if its points can be mapped to nonempty intervals of a
linear order in such a way that $x <_P y$ if and only if the interval
associated to $x$ completely precedes the interval associated to $y$. Thus
interval graphs are the incomparability graphs of interval orders, i.e.\ two
vertices are adjacent in the graph if and only if they are incomparable in
the partial order.\smallskip

Norbert Wiener was probably the first to pay attention to interval orders,
disguised under the less familiar name \lq relations of complete sequence\rq,
in \cite{wiener1914contribution}. Interval graphs and interval orders were
rediscovered and given the current name in \cite{fishburn1970}. There is now
an extensive literature on the topic: \cite{TrotterSurvey} provides a survey
for many result in this area, focusing primarily on finite
structures.\smallskip

Interval graphs and interval orders are extensively employed in diverse
fields like psychology, archaeology and physics, just to mention a few.
Wiener himself noticed that interval orders are useful for the analysis of
temporal events and in the representation of measures subject to a margin of
error. Interval orders actually occur in many digital calendars, where hours
and days form a linear order and a rectangle covers the time assigned to an
appointment: if two rectangles intersect, we better choose which event we
will miss. Intervals are also suitable for representations of measurements of
physical properties which are subject to error, since they can take into
account the accuracy of the measuring device much better than a
representation with points. In psychology and economics the overlap between
two intervals often indicates that the corresponding stimuli or preferences
are indistinguishable.\smallskip

In the first paragraph we described how to build an interval order from an
interval representation of an interval graph. In general, an interval graph
leads to many different interval orders on its vertices: an extreme example
is a totally disconnected graph which is associated to any total order on its
vertices. This paper deals with the situation were the interval graph is
\emph{uniquely orderable}, i.e.\ there is essentially only one interval order
associated to the given interval graph. (The \lq\lq essentially\rq\rq\ in the
previous sentence is due to the obvious observation that if an interval order
is associated to a graph, then the same is true for the reverse partial
order.) Here the extreme example is a complete graph, which is associated to
a unique partial order, the antichain of its vertices.\smallskip

The question of which interval graphs are uniquely orderable is easily
settled for non connected graphs. It is in fact immediate that a non
connected interval graph is uniquely orderable if and only if it has at most
two components each of which is complete.

We can thus restrict our attention to connected interval graphs. In this
context, Fishburn \cite[\S3.6]{Fishburn}, building on results proved in
\cite{Hanlon82}, provides two characterizations of unique orderability for
finite graphs. Indeed, some steps of the proof heavily rely on the finiteness
of the graph. This is in contrast with the rest of Fishburn's monograph,
where results are systematically proved for arbitrary interval graphs and
orders; we thus believe that Fishburn did not know whether his result held
for infinite interval graphs as well. The main result of this paper solves
this issue by extending Fishburn's characterizations to arbitrary interval
graphs by an entirely different proof (for undefined notions see
\S\ref{Sec:prelim} below):

\begin{mtheorem}\label{MainTheorem}
Let $(V,E)$ be a (possibly infinite) connected interval graph. Let
$W=\{(a,b)\in V \times V \mid \neg a \E b\}$ and $(a,b) \Q (c,d) \iff a \E c
\land b \E d$. The following are equivalent:
\begin{enumerate}
\item $(V,E)$ is uniquely orderable;
\item $(V,E)$ does not contain a buried subgraph;
\item the graph $(W,Q)$ has two components.
% $(W,Q)$ is the graph defined in \ref{defQ0}.
\end{enumerate}
\end{mtheorem}

Fishburn's statement is slightly different from ours, since it is formulated
for connected interval graphs without universal vertices. Since universal
vertices (i.e.\ those adjacent to all vertices of the graph) are incomparable
to all other vertices in any partial order associated to an interval graph,
removing all universal vertices does not change the unique orderability of
the graph. We prefer our formulation of the result since it highlights the
connectedness of the graph, which is the central property characterising the
class of interval graphs for which \cref{MainTheorem} holds.

%A \emph{universal vertex} in a graph $(V,E)$ is a vertex which is adjacent to all the vertices of the graph. Notice that a universal vertex is incomparable with all other elements of $V$ in any order associated to $(V,E)$. Therefore. Thus when studying the graphs which are uniquely orderable we can restrict ourselves to graphs without universal vertices.

A typical method to lift a result from finite structures to arbitrary ones is
compactness. Hence, once \cref{MainTheorem} is proved for finite interval
graphs, the first attempt to generalise it to the infinite is to argue by
compactness. This is not obvious and, using tools from mathematical logic, we
are able to show that it is in fact impossible. To this end we work in the
framework of reverse mathematics, a research program whose goal is to
establish the minimal axioms needed to prove a theorem. In this framework
compactness is embodied by the formal system $\wkl$. We first indicate, with
results which parallel those obtained in \cite{marcone2007} about interval
orders, that all the basic aspects of the theory of interval orders can be
developed in $\wkl$. On the other hand we prove the following:

\begin{mtheorem}\label{MainTheoremRM}
Over the base system $\rca$, the following are
equivalent:
\begin{enumerate}
\item $\aca$,
\item a countable connected interval graph $(V,E)$ is uniquely orderable if
    and only if does not contain a buried subgraph.
\end{enumerate}
\end{mtheorem}

Since $\aca$ is properly stronger than $\wkl$ this shows that compactness
does not suffice to prove \cref{MainTheorem}.\bigskip
%
%A preliminary step, in order to obtain the previous result, is to inspect and
%clarify the definition of interval graphs. In the literature there are
%several equivalent characterizations, which nonetheless may not be equivalent
%in each systems of axioms. \cite{marcone2007} already carried out a similar
%analysis for interval orders and \cref{ChapterInterGraph,ChapterIndifGraph}
%follow the line of that article\footnote{Damir Dzhafarov in
%\cite{Dzhafarov11} studied the notion of saturated orders, a generalisation
%of interval orders introduced by Reinhard Suck only for finite posets.
%Dzhafarov extended this notion to infinite saturated orders and studied it
%within subsystems of second order arithmetic.}.
%
%This analysis provides another motivation, besides the combinatorial one,
%which leads us to prove \cref{MainTheorem} and \cref{MainTheoremRM}. In fact,
%the interval graph defined in \cref{order2} is not uniquely orderable and
%some of its associated orders are definable in $\rca$, the base system of
%reverse mathematics, while others  require $\wkl$. We exploited this fact to
%prove  \cref{order2}.

\cref{Sec:prelim} establishes notation and terminology, while
\cref{Sec:UniquelyOrderable} is devoted to the proof of \cref{MainTheorem}.
\cref{Sec:ReverseMath} gives an overview of the reverse mathematics of
interval graphs: the first author's PhD thesis \cite{MFCthesis} includes full
proofs. The last section is devoted to the proof of \cref{MainTheoremRM}.

\section{Preliminaries}\label{Sec:prelim}

In this section we establish the terminology used in the paper and underline
some properties of interval graphs that turn out to be useful in the next
section.

All the graphs $(V,E)$ in this paper are such that $E \subseteq V \times V$
is a symmetric relation (we do not ask $E$ to be irreflexive, as in some
cases it is convenient to have reflexivity). As usual, we write $v \E u$ to
mean $(v,u) \in E$ and, if $V' \subseteq V$, we write $(V',E)$ in place of
$(V', E \cap (V' \times V'))$. We denote by $(V,\overline{E})$ the
\emph{complementary graph} of $(V,E)$: for $u,v \in V$ we have $u
\,\overline{E}\, v$ if and only if $u \E v$ does not hold.

Paths and cycles are defined as usual, and their length is the number of
their edges.
%Let $(V,E)$ be a graph. A \emph{path} in $(V,E)$ is a sequence $v_0, \dots,
%v_n$, for some $n \in \nat$, of elements of $V$ such that $v_i \E v_{i+1}$
%for each $i < n$. The length of the path $v_0, \dots, v_n$ is $n$.
%A \emph{cycle} in $(V,E)$ is a path with $v_n = v_0$.
A \emph{simple cycle} $v_0 \E \dots \E v_n$ is a cycle such that the vertices
in $v_0, \dots, v_{n-1}$ are distinct.  A \emph{chord} of a cycle $v_0 \E v_1
\E \dots \E v_n$ is an edge $(v_i,v_j)$ with $2 \le j-i \le n-2$. The chord
is \emph{triangular} if either $j-i=2$ or $j-i=n-2$.

\begin{definition}\label{DefComparability}
If $(V,\prec)$ is a strict partial order, the \emph{comparability graph of
$(V,\prec)$} is the graph $(V,E)$ such that for $v,u\in V$ it holds that $v
\E u$ if and only if either $v\prec u$ or $u\prec v$. The
\emph{incomparability graph of $(V,\prec)$} is the complementary graph of the
comparability graph, so that two  vertices are adjacent if and only
if they coincide or are $\prec$-incomparable.
%
%A graph $(V,E)$ is a \emph{comparability graph} if there exists a partial
%order $(V,\prec)$ such that for each vertices $v,u\in V$ it holds that $v \E
%u$ if and only if either $v\prec u$ or $u\prec v$. An \emph{incomparability
%graph} is the complementary graph of a comparability graph.
%
%If $(V,E)$ is the comparability or incomparability graph of the partial order
%$(V,\prec)$ we say that $(V,\prec)$ is one of the \emph{orders associated} to
%$(V,E)$.
\end{definition}

While the comparability graph of a strict partial order is irreflexive, its
incomparability graph is reflexive.

Notice that a graph $(V,E)$ can be the incomparability graph of more than one
partial order: we say that each such partial order is \emph{associated to
$(V,E)$}. In particular, $\prec$ and the dual of $\prec$ (i.e.\ $\prec'$ such
that $u \prec' v$ iff $v \prec u$) are associated to the same incomparability
graph.

\begin{definition}
A  graph $(V,E)$ is \emph{uniquely orderable} if it is the incomparability
graph of a partial order $\prec$ and the only other partial order associated
to $(V,E)$ is the dual order of $\prec$; in other words, there exists a
unique (up to duality) partial order $\prec$ such that for each $v,u \in V$
it holds that $\neg u \E v$ if and only if $u \prec v$ or $v \prec u$.
\end{definition}

The following definition formalises the intuitive idea of interval graph
given in the previous pages.

%\begin{definition} \label{defClosedIntgraph}
%A graph $(V,E)$ is a \emph{closed interval graph} if it is reflexive and
%there exists a linear order $(L,<_L)$ and functions $f_R,f_L\colon V\to L$
%such that for all $v,u \in V$ we have $f_L(v) <_L f_R(v)$ and, setting $F(v)
%= \{\ell \in L \mid f_L(v) \le_L \ell \le_L f_R(v)\}$, we have
%\[
%v \E u \Leftrightarrow F(v) \cap F(u) \neq \emptyset.
%\]
%We say that $(L,<_L,f_R,f_L)$ (but often only $(f_R,f_L)$ or just $F$) is a
%\emph{representation} of $(V,E)$.
%\end{definition}

\begin{definition} \label{defClosedIntgraph}
A graph $(V,E)$ is an \emph{interval graph} if it is reflexive and there
exist a linear order $(L,<_L)$ and a map $F \colon V \to \wp(L)$ such that
for all $v,u \in V$, $F(v)$ is an interval in $(L,<_L)$ (i.e.\ if $\ell <_L
\ell' <_L \ell''$ and $\ell,\ell'' \in F(v)$, then also $\ell' \in F(v)$) and
\[
v \E u \Leftrightarrow F(v) \cap F(u) \neq \emptyset.
\]

It is well-known that we may in fact assume that there exist functions
$f_L,f_R\colon V\to L$ such that $F(v) = \{\ell \in L \mid f_L(v) \le_L \ell
\le_L f_R(v)\}$ for all $v \in V$ (this is the definition given in
\cite{Fishburn}).

We say that $(L,<_L,f_L,f_R)$ (but often only $(f_L,f_R)$ or just $F$) is a
\emph{representation} of $(V,E)$.
\end{definition}

To decide whether two vertices $u$ and $v$ are adjacent in an interval graph
with representation $(f_L,f_R)$ we can assume without loss of generality that
$f_L(v) \le_L f_L(u)$ and then simply check whether $f_L(u) \le_L f_R(v)$.

In the context of a representation $(f_L,f_R)$ of an interval graph, we write
$F(v) <_L F(u)$ in place of $f_R(v) <_L f_L(u)$. Then $\neg v \E u$ means
that either $F(v) <_L F(u)$ or $F(u) <_L F(v)$.

%Notice that the representation $F$ of an interval graph induces an associated
%partial order by setting $v \prec u$ if and only if $F(v) <_L F(u)$. Vice
%versa, if $\prec$ is a partial order associated to the interval graph $(V,E)$
%we can find a representation $F$ of $(V,E)$ such that $v \prec u$ if and only
%if $F(v) <_L F(u)$.

Figure \ref{FigIntGraph} provides an example of interval graph, while the
graph in Figure \ref{FigNonIntGraph} does not have an interval
representation (in the figures self loops are not shown for clarity).	
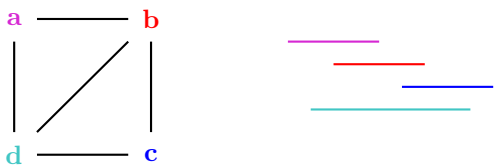
\begin{figure}
\begin{centering}
\begin{tikzpicture}[scale=.3]
\node[viola] at (-8,4){\textbf{a}};
\node[red] at (-2,4){\textbf{b}};
\node[blue] at (-2,-2){\textbf{c}};
\node[azzurro] at (-8,-2){\textbf{d}};
\draw[thick] (-7,4) -- (-3,4);
\draw[thick] (-7,-2) -- (-3,-2);
\draw[thick] (-2,3) --(-2,-1);
\draw[thick] (-8,3) -- (-8,-1);
\draw[thick] (-7,-1) -- (-3,3);
%\filldraw[black] (0,0) circle (2pt) node[anchor=west] {Intersection point};

\draw[viola,thick] (4,3) -- (8,3);
\draw[red,thick]  (6,2) -- (10,2);
\draw[blue,thick] (9,1) -- (13,1);
\draw[azzurro,thick]  (5,0) -- (12,0);
\end{tikzpicture}
\caption{\small An example of interval graph with its representation}
\label{FigIntGraph}
\end{centering}
\end{figure}

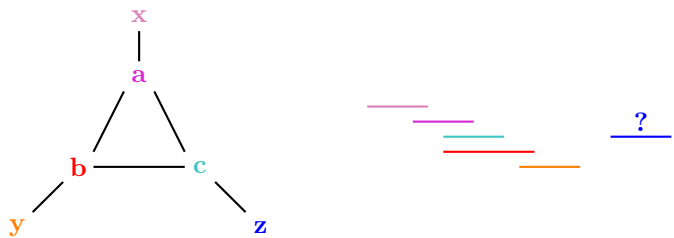
\begin{figure}
\begin{centering}
\begin{tikzpicture}[scale=.2]
\node[thick,viola] at (-6,6){\textbf{a}};
\node[thick,azzurro] at (-2,0){\textbf{c}};
\node[red] at (-10,0){\textbf{b}};
\node[oliva] at (-6,10){\textbf{x}};
\node[orange] at (-14,-4){\textbf{y}};
\node[blue] at (2,-4){\textbf{z}};
\node[blue] at (27,3){\textbf{?}};
\draw[thick] (-9,0) -- (-3,0);  %bc
\draw[thick] (-9,1) -- (-7,5);  %ba
\draw[thick] (-3,1) --(-5,5);  %ca
\draw[thick] (-6,7) -- (-6,9); %ax
\draw[thick] (-13,-3) -- (-11,-1); %yb
\draw[thick] (1,-3) -- (-1,-1); %zc
%\filldraw[black] (0,0) circle (2pt) node[anchor=west] {Intersection point};

\draw[oliva,thick]  (9,4) -- (13,4);
\draw[viola,thick] (12,3) -- (16,3);
\draw[azzurro,thick]  (14,2) -- (18,2);
\draw[red,thick] (14,1) -- (20,1);
\draw[orange,thick] (19,0)--(23,0);
\draw[blue,thick] (25,2)--(29,2);
\end{tikzpicture}
\caption{\small A graph which is not an interval graph, with a partial representation}
\label{FigNonIntGraph}
\end{centering}
\end{figure}

A classical characterization of interval graphs is the following
(\cite{lekkeikerker}, see \cite[Theorem 3.6]{Fishburn}).

\begin{definition}\label{def:triangulated}
A graph $(V,E)$ is \emph{triangulated} if every simple cycle of length four
or more has a chord. An \emph{asteroidal triple} in $(V,E)$ is an independent
set of three vertices (i.e.\ a set of pairwise non adjacent vertices) of $V$
such that any two of them are connected by a path that avoids the vertices
adjacent to the third.
\end{definition}

\begin{theorem}\label{char_ig}
A reflexive graph $(V,E)$ is an interval graph if and only it is triangulated and has
no asteroidal triples.
\end{theorem}

\begin{proposition}\label{SubPath}
Let $v_0 \E \dots \E v_n$ be a path in the
interval graph $(V,E)$ with representation $F$, and suppose $w \in V$ is such
that $F(w) \nless_L F(v_0)$ and $F(v_n) \nless_L F(w)$. Then $v_i \E w$ for
some $i \le n$, and hence $v_0 \E \dots \E v_i \E w$ and $w \E v_i \E \dots
\E v_n$ are paths.
\end{proposition}
\begin{proof}
Let $i \le n$ be maximum such that $F(w) \nless_L F(v_i)$. If $i=n$, then
$F(w) \nless_L F(v_n)$ and $F(v_n) \nless_L F(w)$ imply $v_n \E w$. If $i<n$,
then $F(w) <_L F(v_{i+1})$ and $F(v_i) \nless_L F(v_{i+1})$ (because $v_i \E
v_{i+1}$) imply $F(v_i) \nless_L F(w)$. This, together with $F(w) \nless_L
F(v_i)$, yields $v_i \E w$.
\end{proof}

\begin{definition}\label{DefMinPath}
Let $(V,E)$ be a graph. A path $v_0 \E \dots \E v_n$ is a \emph{\minpath} if
$\neg v_i \E v_j$ for every $i,j$ such that $i+1 < j \leq n$.
\end{definition}

Notice that if $v_0 \E \dots \E v_n$ is a path of minimal length among the
paths connecting $v_0$ and $v_n$, then it is a \minpath, but the reverse
implication does not hold. %\cref{FigMinPath} shows an example of a minimal
%path and of a non minimal path in an interval graph (depicted by its interval
%representation).
%
%\begin{figure}[h]
%\begin{center}
%\begin{tikzpicture}[scale=.5]
%\draw[|-|] (0,1) --  node[above=0mm] {\small $v_0$} (2,1);
%\draw[|-|] (4,1) -- node[above=0mm] {\small $v_3$} (6,1);
%\draw[|-|] (1,0) --node[above=0mm] {\small $v_1$} (3.5,0);
%\draw[|-|] (3,-1) -- node[above=0mm] {\small $v_2$} (7,-1);
%\draw[|-|] (-0.8,-1) -- node[above=0mm] {\small $v_4$} (1.2,-1);
%\draw[|-|] (-0.5,0) -- node[above=0mm] {\small $v_5$} (-2.5,0);
%\end{tikzpicture}
%\caption{\small $v_0,v_1,v_2,v_3$ is a \minpath, while $v_0,v_4,v_1,v_2,v_3$ is not.}
%\label{FigMinPath}
%\end{center}
%\end{figure}

\begin{property} \label{Refine}
Let $(V,E)$ be a graph. Then each path can be refined to a \minpath.	
\end{property}
\begin{proof}
The statement follows immediately from the following observation: if $v_0 \E
\dots \E v_n$ is a path and $v_i \E v_j$ with $i+1 < j \leq n$, then $v_0 \E
\dots \E v_i \E v_j \E \dots \E v_n$ is still a path. 	
\end{proof}

\begin{property}\label{PropMinpath}
Let $(V,E)$ be an interval graph with representation $(L,<_L,f_L,f_R)$ and
suppose that $v_0 \E \dots \E v_n$ is a \minpath\ with $F(v_0) <_L F(v_n)$.
\begin{enumerate}[label=(\roman*)]
 \item Then $f_R(v_i) <_L f_R(v_{i+1})$ for each $i < n-1$ and $f_L(v_j)
     <_L f_L(v_{j+1})$ for each $j > 0$;
 \item if $F(v) <_L F(v_0)$, then $\neg v_i \E v$ for every $i \neq 1$;
     symmetrically, if $F(v_n) <_L F(v)$, then $\neg v_i \E v$ for every $i
     \neq n-1$.
% \item suppose that $F(v) <_L F(v_0) <_L F(v_n) <_L F(u)$, for some
%     $v,u,v_0,v_n \in V$, and that for every path $v_0 \E \dots \E v_n$
%     there exist $i,j \le n$ such that $v_i \E v$ and $v_j \E u$. Then for
%     every \minpath\ $v_0 \E \dots \E v_n$ we have: $v_i \E v$ if and only
%     if $i = 1$, and $v_j \E u$ if and only if $j = n-1$.
\end{enumerate}
\end{property}
\begin{proof}
To check the first conjunct of (i), suppose $i < n-1$ is least such that
$f_R(v_{i+1}) \le_L f_R(v_i)$. Since $i < n-1$ it holds that $\neg v_j \E
v_n$ for each $j \le i$ by definition of \minpath. An easy induction,
starting with our assumption $F(v_0) <_L F(v_n)$, shows that $F(v_k) <_L
F(v_n)$ for each $k \le i$. Thus, in particular it holds that $f_R(v_i) <_L
f_R(v_n)$. Let $m\le n$ be least such that $f_R(v_i) <_L f_R(v_m)$ and notice
that $m>i+1$ by choice of $i$. By choice of $m$ it holds that $f_R(v_{m-1})
\le_L f_R(v_i) <_L f_R(v_m)$, and so that $f_L(v_{m}) \le_L f_R(v_{m-1})$
because $v_{m-1}\E v_m$. To summarise we get that $f_L(v_{m}) \le_L f_R(v_i)
<_L f_R(v_m)$, namely that $v_i \E v_m$ contrary to the definition of
\minpath.

The second conjunct of (i) follows from the first considering the interval
representation given by the linear order $(L,>_L)$ and by the maps $f_L$ and
$f_R$.\smallskip

For (ii), let $v_0 \E \dots \E v_n$ be a \minpath\ and  $F(v) <_L F(v_0) <_L
F(v_n)$. Assume $v \E v_i$, for some $i > 1$ (notice that $\neg v \E v_0$ by
assumption). Since $f_R(v) <_L f_L(v_0)$ by assumption, $f_R(v_0) <_L
f_R(v_i)$ by (i), and  $f_L(v_i) <_L f_R(v)$ by $v \E v_i$, it holds that
$v_0 \E v_i$, contrary to the definition of \minpath.
%Since for each path
%$v_0 \E \dots \E v_n$ there exist $i\le n$ such that $v_i \E v$ by
%assumption, it must be $v_1 \E v$. One can argue analogously for $u$.
\end{proof}

\section{Uniquely orderable connected interval graphs} \label{Sec:UniquelyOrderable}

In this section we prove \cref{MainTheorem}. Suppose $(V,E)$ is a connected
incomparability graph. Saying that $(V,E)$ is not uniquely orderable amounts
to check that there are two partial orders $\prec$ and $\prec'$ associated to
$(V,E)$ and three vertices $a,b,c \in V$ such that $a \prec b \prec c$ and $b
\prec' a \prec' c$. The vertices $a$ and $b$ can be reordered regardless, so
to speak, the order of $c$. The connected graph pictured (by one of its
interval representations) in Figure \ref{Fig3UniqOrd} is an example of a non
uniquely orderable connected interval graph (in fact the intervals for $a$
and $b$ can be swapped without changing their relationship with the intervals
$c$ and $k$).
\begin{figure}
\begin{centering}
  \begin{displaymath}
\xymatrix@=10pt@R=1pt@C=3pt@M=10pt{ &\ar@{-}[rr]^-{a} && \ar@{-}[rr]^-{b} && \ar@{-}[rr]^-{c} &&   \\
                       \ar@{-}[rrrrrr]^-{k} &&&&&&  &&}
\end{displaymath}
\caption{\small An interval representation of a non uniquely orderable connected interval graph} \label{Fig3UniqOrd}
\end{centering}
\end{figure}
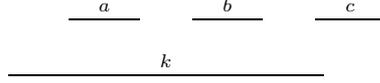

The first characterization of uniquely orderable interval graphs exploits the
above observation to identify subgraphs which are forbidden in uniquely
orderable interval graphs.

\begin{definition} \label{defbburied}
Let $(V,E)$ be a graph. For $B \subseteq V$ let $K(B)=\{v \in V \mid \forall
b\in B \, (v\E b)\}$ and $R(B)=V \setminus (B \cup K(B))$. We say that $B$ is
a \emph{buried subgraph} of $(V,E)$ if the following hold:
\begin{enumerate}[label=(\roman*)]
\item there exist $a,b\in B$ such that $\neg a \E b$,
\item $K(B) \cap B = \emptyset$ and $R(B) \ne \emptyset$,
\item if $b \in B$ and $r \in R(B)$, then $\neg b \E r$. %if $v_0 \E \dots \E v_n$ is a path such that $v_0 \in B$ and $v_n \in V \setminus B$, then there exists $i \le n$ such that $v_i \in K(B)$.
\end{enumerate}
\end{definition}

The last point in the previous definition implies that any path between a
vertex in $B$ and a vertex outside $B$ must go through a vertex in $K(B)$.
The main consequence of (iii), which we use many times without mention, is
that if $v \in V$ is such that there exist $a,b \in B$ such that $v \E a$ and
$\neg v \E b$, then $v \in B$ (because $\neg v \E b$ implies $v \notin K(B)$,
while $v \E a$ and (iii) imply $v \notin R(B)$).

Our definition of buried subgraph is slightly different from the one in
\cite{Fishburn}, but it is equivalent for the class of graphs studied by
Fishburn, i.e.\ connected interval graphs without universal vertices. Since
we allow universal vertices, in condition (ii) we substituted $K(B) \ne
\emptyset$ with $R(B) \ne \emptyset$ (the former condition implies the latter
if there are no universal vertices, the reverse implication holds if the
graph is connected by (iii)). Moreover we restated condition (iii) in
simpler, yet equivalent, terms.

The other main character of \cref{MainTheorem} is the graph $(W,Q)$.

\begin{definition}
If $(V,E)$ is a graph we let $W=\{(a,b) \in V \times V \mid \neg a \E b\}$
and (writing $ab$ in place of $(a,b)$ for concision) $ab \Q cd$ if and only
if $a \E c$ and $b \E d$.

If $ab$ and $cd$ are elements of $W$ which are connected by a path in $(W,Q)$
we write $ab \Qcon cd$.
\end{definition}

\begin{proposition} \label{PropertyQ}
Let $(V,E)$ be an interval graph and $\prec$ a partial order associated to
$(V,E)$. If $ab, cd \in W$ are such that $ab \Qcon cd$ and $a \prec b$, then
$c \prec d$. In particular we have $\neg ab \Qcon ba$.
\end{proposition}
\begin{proof}
Suppose first that $ab \Q cd$, so that $a \E c$ and $b \E d$. Notice that $b
\E c$ and $a \E d$ cannot both hold. In fact, if $a,b,c,d$ are not all distinct, then this would contradict $ab \in W$ or $cd \in W$. Otherwise,  $a \E c \E b \E d \E a$ would
be a simple cycle of length four without chords, against \cref{char_ig}. If
$\neg b \E c$, then $c \prec b$ because $a \prec b$ and $a \E c$. From this
we obtain $c \prec d$, since $d \E b$. If instead $\neg a \E d$ we obtain
first $a \prec d$ and then again $c \prec d$.

To derive $c \prec d$ from $ab \Qcon cd$ it suffices to apply the
transitivity of $\prec$ to a $Q$-path connecting $ab$ with $cd$.
\end{proof}

The last part of the previous proposition implies that if $W \neq \emptyset$
(which is equivalent to $(V,E)$ being not complete), then $(W,Q)$ has at least
two components. Moreover, if $(W,Q)$ has more than two (and so at
least four) components, then for every partial order $\prec$
associated to $(V,E)$ there exist $ab, cd \in W$ such that $a \prec b$ and $c
\prec d$, yet $ab \Qcon cd$ fails.\medskip

We split, as originally done by Fishburn, the proof of \cref{MainTheorem} in
three steps corresponding to (1) implies (2) (\cref{tuniqord}), (3) implies
(1) (\cref{QuniqordRCA}), and (2) implies (3) (\cref{SufCondition}). The
proof of the first implication in \cite{Fishburn} is not completely accurate,
and we apply Fishburn's idea after a preliminary step which is necessary even
when the graph is finite. The second implication is straightforward and
applies to interval graphs of any cardinality. The proof of the last
implication is completely new and requires more work.

The connectedness of the graph is not needed in the first two implications.
Moreover, the hypotheses of \cref{tuniqord} could be further relaxed, as the
proof applies to arbitrary incomparability graphs.

\begin{lemma}\label{tuniqord}
Every uniquely orderable interval graph does not contain a buried subgraph.
\end{lemma}
\begin{proof}
Let $(V,E)$ be an interval graph with a buried subgraph $B$. Fix a partial
order $\prec_0$ associated to $(V,E)$ and some $b_0 \in B$. We define a new
binary relation $\prec$ on $V$ as follows: when either $u,v \in B$ or $u,v
\notin B$ set $u \prec v$ if and only if $u \prec_0 v$; when $b \in B$ and $v
\notin B$ set $b \prec v$ if and only if $b_0 \prec_0 v$, and $v \prec b$ if
and only if $v \prec_0 b_0$. Thus the whole $B$ is $\prec$-above the elements
not in $B$ which are $\prec_0$-below $b_0$ and $\prec$-below the elements not
in $B$ which are $\prec_0$-above $b_0$.

Using the fact that the vertices not in $B$ are either $\prec_0$-incomparable
to every vertex of $B$ or $\prec_0$-comparable to every vertex of $B$, it is
straightforward to check that $\prec$ is transitive, and hence a partial
order. For the same reason $\prec$ is associated to $(V,E)$. The key feature
of $\prec$ (not necessarily shared by $\prec_0$) is that $B$ is
$\prec$-convex, i.e.\ if $b \prec v \prec b'$ with $b,b' \in B$, then $v \in
B$ as well. Indeed, if $v \notin B$, then $b \prec v$ implies $b_0 \prec_0 v$
and $v \prec b'$ implies $v \prec_0 b_0$.

Following now \cite{Fishburn}, let $\prec'$ be such that the restrictions of
$\prec$ and $\prec'$ to $B$ are dual, while $\prec'$ and $\prec$ coincide on
$V \setminus B$ and between elements of $B$ and $V \setminus B$. Formally, $u
\prec' v$ if and only if either $u,v \in B$ and $v \prec u$, or if at least
one of $u$ and $v$ does not belong to $B$ and $u \prec v$. The transitivity
of $\prec'$ is a consequence of the $\prec$-convexity of $B$ (an observation
lacking in the proof given in \cite{Fishburn}) and hence $\prec'$ is a
partial order associated to $(V,E)$.

If $x,y \in B$ are such that $x \prec y$ and $v \in R(B)$ (these elements
exist by \cref{defbburied}) we have either $x \prec y \prec v$ or $v \prec x
\prec y$. In the first case $y \prec' x \prec' v$, in the second case $v
\prec' y \prec' x$, witnessing that $\prec'$ is neither $\prec$ nor the dual
order of $\prec$.
\end{proof}

\begin{lemma}\label{QuniqordRCA}
Let $(V,E)$ be an interval graph. If ($W,Q$) has two components,
then $(V,E)$ is uniquely orderable.
\end{lemma}
\begin{proof}
This follows easily from \cref{PropertyQ}.
\end{proof}

For the proof of \cref{SufCondition} we describe a construction that,
starting from a pair of non-adjacent vertices, attempts to build the minimal
buried subgraph containing those two vertices. We then show that if this
attempt always fails, then for any $ab, cd \in W$ either $ab \Qcon cd$ or $ab
\Qcon dc$.

\begin{construction}\label{ConstrB}
Let $(V,E)$ be a connected interval graph and $v,u \in V$ be such that $\neg v \E
u$. We define recursively $B_n(v,u) \subseteq V$:
\begin{align*}
 B_0(v,u) &=\{v,u\}\\
 B_{n+1}(v,u) &=\{w \in V \mid \exists z,z' \in B_n(v,u) \, (z\E w \land \neg z'\E w) \}
\end{align*}
We then set $B(v,u)=\bigcup B_n(v,u)$. If $w \in B(v,u)$ let $e_w$ be the
least $n$ such that $w \in B_n(v,u)$ (formally we should write $e_w^{v,u}$
but we omit the superscript as $v$ and $u$ will always be understood).
\end{construction}

A straightforward induction shows that $B_n(v,u)\subseteq B_{n+1}(v,u)$, for
each $n \in \nat$ (for the base step recall that interval graphs are
reflexive, so that $v$ and $u$ themselves witness that $v,u \in B_1(v,u)$).

We now show that $B(v,u)$ is close to being a buried subgraph.

\begin{property} \label{SemiBuried}
In the situation of \cref{ConstrB}, $B(v,u)$ is a buried subgraph if and only
if $R(B(v,u)) \ne \emptyset$.
\end{property}
\begin{proof}
Notice that Condition (i) of \cref{defbburied} is witnessed by $v$ and $u$.
%Suppose $v_0 \in B(v,u)$ and $v_n \notin B(v,u)$. Since $(V,E)$ is connected, let $v_0, \dots, v_n$ be a path and suppose $i$ is maximum such that $v_i \in B(v,u)$. Since $v_{i+1} \notin B(v,u)$ even if $v_{i+1} \E v_i$, it holds that there is no $z \in B(v,u)$ such that $\neg v_{i+1} \E z$, namely that $v_{i+1} \in K(B(v,u))$.
Condition (3) is obvious, because if $r \notin K(B(u,v))$ but $b \E r$ for
some $b \in B(u,v)$, then $r \in B(u,v)$. Moreover, if $k \in K(B(v,u))$, then
$k \in K(B_n(u,v))$ for every $n$ and hence $k \notin B(u,v)$; hence $B(v,u)
\cap K(B(v,u)) = \emptyset$. Therefore, to verify that $B(v,u)$ is a buried
subgraph it suffices that $R(B(v,u)) \ne \emptyset$.
\end{proof}

In the next propositions, we will always consider a connected interval graph
$(V,E)$ with representation $(L,<_L,f_L,f_R)$, fix $v,u \in V$ with $\neg v
\E u$ and $F(v) <_L F(u)$ and define $B(v,u)$ as in \cref{ConstrB}. For
brevity, we call this set of hypotheses $(\maltese)$ and indicate it next to
the proposition number.

\begin{proposition}[\maltese]\label{Livelli}
Let $x,y \in B(v,u)$. If $F(u) <_L F(x)$ and $f_L(x) \le_L f_L(y)$, then $e_x
\le e_y$. Analogously, if $F(x) <_L F(v)$ and  $f_R(y) \le_L f_R(x)$, then $e_x
\le e_y$ as well.
\end{proposition}
\begin{proof}
We prove the first half of the statement by induction on $e_y$. The base case
is trivially satisfied since there is no $y \in B_0(v,u)$ satisfying the
hypotheses. Assume $e_y>0$ and let $z \in B_{e_y-1}(v,u)$ be such that $z \E
y$. If $f_L(x) \le_L f_L(z)$, then $e_x \le e_z < e_y$ by induction
hypothesis. Otherwise, $f_L(z) <_L f_L(x) \le f_L(y) \le_L f_R(z)$ given that
$z \E y$. This means that $z \E x$, which implies that $x \in B_{e_z+1}(v,u)$
since $u \in B_{e_z}(v,u)$ is such that $\neg u \E x$. Hence $e_x \le e_z+1
\le e_y$

The second half of the statement follows considering the representation
$(L,>_L,f_L,f_R)$.
\end{proof}

\begin{proposition}[\maltese]\label{CorollarioLivelli}
Let $w\in B(v,u)$. If $F(w) \nless_L F(u)$, then there exists a path $u \E b_1
\E \dots b_k \E w$ such that $e_{b_i} < e_w$ for all $i\le k$.

Analogously, if $F(v) \nless_L F(w)$, then there exists a path $v \E b_1 \E
\dots b_k \E w$ such that $e_{b_i} < e_w$ for all $i\le k$.
\end{proposition}
\begin{proof}
By definition of $B_{e_w}(v,u)$ there exists a path $b_0 \E b_1 \E \dots \E
b_k \E w$ where $b_i \in B(v,u)$ and $0=e_{b_0} < e_{b_1} < \dots < e_{b_k} <
e_w$. Hence $b_0 \in \{u,v\}$ and, since $F(w) \nless_L F(u)$ and $F(u)
\nless_L F(v)$, by \cref{SubPath} we can assume that $b_0=u$.

The second half of the statement follows from the first one as usual.
\end{proof}

\begin{proposition}[\maltese]\label{AltroCammino}
Let $x,z \in B(v,u)$ and $m = \max\{e_x,e_z\}$. Assume $F(z) <_L F(x)$ and
$F(v) \nless_L F(x)$ (this implies $m>0$). Then there exists a \minpath\ $z
\E v_1 \E \dots \E v_n \E x$ and $s \in B(v,u)$ with $e_s<m$ such that
$e_{v_i} < m$ and $F(v_i) <_L F(s)$ for each $i \le n$.

Analogously, if $F(x) <_L F(z)$ and $F(x) \nless_L F(u)$ there exists a
\minpath\ $x \E v_1 \E \dots \E v_n \E z$ and $s \in B(v,u)$ with $e_s<m$
such that $e_{v_i} < m$ and $F(s) <_L F(v_i)$ for each $i \le n$.
\end{proposition}
\begin{proof}
Notice that once we find the \minpath\ $z \E v_1 \E \dots \E v_n \E x$ and $s
\in B(v,u)$ with $e_s<m$ such that $e_{v_i} < m$ for all $i \le n$ it
suffices to prove that $F(v_n) <_L F(s)$, since then $F(v_i) <_L F(s)$ for
$i<n$ follows from \cref{PropMinpath}.i.

We can apply \cref{CorollarioLivelli} to both $x$ and $z$ obtaining paths
connecting $v$ to $x$ and $v$ to $z$ and with $e_b<m$ for all vertices $b$,
distinct from $x$ and $z$, occurring in the paths. Joining these paths and
then using \cref{Refine} we obtain a \minpath\ $z \E v_1 \E \dots \E v_n \E
x$ with $e_{v_i}<m$. Notice that $n>0$ as $\neg z \E x$. Let $j < m$ be such
that $e_{v_n}=j$: we may assume $j$ is least for which such a \minpath\
exists.

If $j=0$, then we claim that we can assume $v_n=v$ and hence we can choose
$s=u$. In fact, if $v_n=u$, then $F(x) <_L F(v)$ is impossible and we have $v
\E x$. The hypotheses imply $F(v) \nless_L F(z)$ and, since $F(u) \nless_L
F(v)$, by \cref{SubPath} we can find $i<n$ such that $v \E v_i$ and consider
the path $z \E v_1 \E \dots \E v_i \E v \E x$.

We now assume $j>0$: there exists $s \in B(v,u)$ such that $e_s<j$ and $\neg
s \E v_n$. We claim that $F(v_n) <_L F(s)$, completing the proof. Suppose on
the contrary that $F(s) <_L F(v_n)$ ($F(v_n) \cap F(s) \neq \emptyset$ cannot
hold because $\neg s \E v_n$).

In this case we have $F(s) <_L F(x)$ because $f_L(v_n) <_L f_L(x)$ by
\cref{PropMinpath}.i. Hence $F(v) \nless_L F(s)$ and we can use
\cref{CorollarioLivelli} and \cref{Refine} to obtain a \minpath\ $s \E u_1 \E
\dots \E u_\ell$, with $u_\ell = v$ and $e_{u_i}<e_s$. Since $F(x) \nless_L
F(s)$ and $F(v) \nless_L F(x)$ by \cref{SubPath} there exists $k \le \ell$
such that $u_k \E x$. We distinguish two cases: $F(z) \nless_L F(s)$ and
$F(z) <_L F(s)$.

In the first case we apply \cref{SubPath} to the path $s \E u_1 \E \dots \E
u_k \E x$: there exists $h \le k$ such that $z \E u_h$. Since $z \E u_h \E
\dots \E u_k \E x$ can be refined to a \minpath\ and $e_{u_i}<j$, the
minimality of $j$ is contradicted.

In the second case we apply \cref{SubPath} to the path $z \E v_1 \E \dots \E
v_n \E x$: there exists $h < n$ (recall that $\neg v_n \E s$) such that $v_h
\E s$.  Then $z \E v_1 \E \dots \E v_h \E s \E u_1 \E \dots \E u_k \E x$ can be refined to a
path, which can then be refined to a \minpath\ $z \E w_1 \E \dots \E w_r \E x$.
Notice that $w_r=u_p$, for some $p \le k$ because $\neg v_i \E x$ for every
$i \leq h<n$, by minimality of the path $z \E v_1 \E \dots \E v_n \E x$, and
$F(s) <_L F(x)$. Since $e_{w_r}<j$ we contradict again the minimality of $j$.

The second half of the statement follows from the first one as usual.
\end{proof}

\begin{lemma}[\maltese]\label{lemmaQcammini}
If $x,y \in B(v,u)$, $f_R(x) \le_L f_R(v)$ and $f_L(u) \le_L f_L(y)$, then $vu
\Qcon xy$.
\end{lemma}
\begin{proof}
The proof is by induction on $e_x + e_y$. If $e_x + e_y=0$, then $x=v$ and
$y=u$, so that the conclusion is immediate (recall that $Q$ is reflexive).
Now assume that $e_x + e_y>0$ and suppose $e_x \le e_y$ (if $e_y<e_x$ we can
employ the usual trick of reversing the representation) and hence $e_y>0$.

If $u \E y$, then $xu \Q xy$ and, since the induction hypothesis implies $vu
\Qcon xu$ (because $e_u=0$), we obtain $vu \Qcon xy$. Thus we assume $\neg u
\E y$ and hence $F(u) <_L F(y)$. Let $z \in B_{e_y-1}(v,u)$ be such that $y
\E z$. If $f_L(u) \le_L f_L(z)$, then we can apply the induction hypothesis to
$xz$ obtaining $vu \Qcon xz$. Since $xz \Q xy$, we are done.

We thus assume $f_L(z) <_L f_L(u)$ which, together with $F(u) <_L F(y)$ and
$z \E y$, implies $f_R(u) <_L f_R(z)$ and hence $z \E u$. Notice moreover
that $z \neq u$ and hence (since $z \neq v$ is obvious) $e_y>1$.
%It follows in particular that $vu \Qcon xu$ {\color{red} $\max\{i,0\} \le j$
%e $\min\{i,0\}=0 < i$} and that $vu \Qcon vy$ {\color{red} $\max\{0,j\} = j$
%e $\min\{0,j\}=0 < i$}.
%
% Notice
%that this implies that $F(v) <_L F(y)$. Since  $y \in B_j(v,u)$ there exists
%$w \in B_{j-1}(v,u)$ such that $w \E y$. Two cases are possible: either
%$f_L(u) \le_L f_L(w)$ or $f_L(w) <_L f_L(u)$. If the former holds then the
%induction hypothesis guarantees that $vu \Qcon xw$  {\color{red} Se $i<j$,
%$\max\{i,j-1\} < j$. Se $i=j$, $\max\{i,j-1\} = j$ e $\min\{i,j-1\} = j-1 <
%i$}. Since $xw \Q xy$ holds because $w \E y$ (notice that the assumptions
%imply that $\neg x \E w$), we conclude that $vu \Qcon xy$. If the latter is
%the case, then $u \E w$.
If $\neg x \E z$, then $xu \Q xz \Q xy$ and, since by induction hypothesis
$vu \Qcon xu$, we have $vu \Qcon xy$. If instead $x \E z$ we must have
$f_L(z) \le_L f_R(x) \le_L f_R(v)$. Let $t \in B_{e_y-2}(v,u)$ be such that
$\neg t \E z$. If $F(z) <_L F(t)$, then $F(u) <_L F(t)$ and $f_L(y) <_L
f_L(t)$, so that \cref{Livelli} implies $y \in B_{e_y-2}(v,u)$, which is
impossible. Hence $F(t) <_L F(z)$. This implies $f_R(t) <_L f_R(x)$. It
follows that $x \in B_{e_y-1}(v,u)$, either by \cref{Livelli}, if $F(x)<_L
F(v)$, or because $x \in B_1(v,u)$ if $x \E v$, given that $\neg x \E u$.
Since $vu \Qcon tu$ holds by induction hypothesis and we have also $tu \Q tz
\Q ty$ it suffices to show that $ty \Qcon xy$.

If $t \E x$ the conclusion is immediate, otherwise $F(t) <_L F(x)$. Since
$F(v) \nless_L F(x)$ we can apply \cref{AltroCammino} finding a \minpath\ $t
\E u_1 \E \dots \E u_n \E x$ and $s \in B_{e_y-2}(v,u)$ such that $u_i \in
B_{e_y-2}(v,u)$ and $F(u_i) <_L F(s)$ for all $i \le n$.
  We claim that $\neg u_i \E y$, for each $i \le n $, so that $ty \Q u_1y \Q \dots \Q u_ny \Q
xy$ witnesses $ty \Qcon xy$. Indeed, if $u_i \E y$, for some $i\le n$, we would have $f_L(y) <_L
f_R(u_i) <_L f_L(s)$ and we could apply \cref{Livelli} to obtain $y \in
B_{e_y-2}(v,u)$, which is impossible.
\end{proof}

\begin{theorem} \label{SufCondition}
Let $(V,E)$ be a connected interval graph. If $(V,E)$ does not contain a
buried subgraph, then $(W,Q)$ has two components.
\end{theorem}
\begin{proof}
Fix a representation $(L,<_L,f_L,f_R)$ of $(V,E)$ and assume that $(V,E)$
does not contain a buried subgraph. We show that if $ab,cd \in W$ are such
that $F(a) <_L F(b)$ and $F(c) <_L F(d)$, then $ ab \Qcon cd$. We can
assume without loss of generality that $f_R(c) \le_L f_R(a)$. We consider
three cases:
\begin{enumerate}[label=\textsf{Case \arabic*:}]
\item $f_L(b) <_L f_L(d)$: $B(a,b)$ (which satisfies the hypotheses of
    $(\maltese)$) is not a buried subgraph and hence by \cref{SemiBuried}
    we must have $B(a,b) = V$. In particular $c, d \in B(a,b)$ and we are
    in the hypotheses of \cref{lemmaQcammini}: we conclude that $ab \Qcon
    cd$.
\item $f_R(a) <_L f_L(d)\le_L f_L(b)$: $B(a,d)$ (which satisfies the
    hypotheses of $(\maltese)$) is not a buried subgraph and hence by
    \cref{SemiBuried} $b,c \in B(a,d)$. \cref{lemmaQcammini} implies both
    $ad \Qcon ab$ and $ad \Qcon cd$. It follows that $ab \Qcon cd$.
\item $f_L(d) \le_L f_R(a)$: neither $B(a,b)$ nor $B(c,d)$ (which both
    satisfy the hypotheses of $(\maltese)$) is a buried subgraph. By
    \cref{SemiBuried} we have $c \in B(a,b)$, which implies $ab \Qcon cb$,
    and $b \in B(c,d)$, which together with $f_L(d) <_L f_L(b)$ yields $cd
    \Qcon cb$ (we use \cref{lemmaQcammini} in both cases). Thus $ab \Qcon
    cd$ also in this case.\qedhere
\end{enumerate}
\end{proof}

\section{Reverse mathematics and interval graphs}\label{Sec:ReverseMath}

Reverse mathematics is a research program, which dates back to the Seventies,
whose goal is to find the exact axiomatic strength of theorems from different
areas of mathematics. It deals with statements about countable, or countably
representable, structures, using the framework of the formal system of second
order arithmetic $\Z_2$. We do not introduce reverse mathematics here, but
refer the reader to monographs such as \cite{sosoa} and \cite{Hirschfeldt15}.

The subsystems of second order arithmetic are obtained by limiting the
comprehension and induction axioms of $\Z_2$ to specific classes of formulas.
We mention only the subsystems we are going to use in this paper: $\rca$ is
the weak base theory corresponding to computable mathematics, $\wkl$ extends
$\rca$ by adding Weak K\"{o}nig's Lemma (each infinite binary tree has an
infinite path), and $\aca$ is even stronger allowing for definitions of sets
by arithmetical comprehension. It is well-known that $\wkl$ is equivalent to
many compactness principles and thus we can claim that a theorem not provable
in $\wkl$ does not admit a proof by compactness. In particular this applies
to \cref{MainTheorem}, as \cref{MainTheoremRM} shows that it is not provable
in $\wkl$.

%We recall the following well known equivalences (see \cite[Lemma
%IV.4.4]{sosoa} and ???), which are used in the paper to prove reversals.
%
%\begin{theorem}[$\rca$]
%$\wkl$ is equivalent to the following statement: for each injective functions
%$f, g \colon \Nb\to\Nb$ such that $\ran(f) \cap \ran(g) = \emptyset$ there
%exists a set $X$ such that $\ran(f) \subseteq X$ and $\ran(g) \cap X =
%\emptyset$.
%\end{theorem}
%
%\begin{definition}
%Let $f \colon \nat \to \nat$ be an injective function. The number $n \in
%\nat$ is a true stage if $f (k) > f (n)$ for each $k > n$. Otherwise, $n$ is
%a false stage.	
%\end{definition}
%
%\begin{theorem}[$\rca$] \label{ReversalACA}
%$\aca$ is equivalent to the following statement: for each injective function
%$f\colon \Nb\to\Nb$ there exists an infinite subset of the true stages of
%$f$.
%\end{theorem}

The second author studied the equivalence of different characterizations of
interval orders from the reverse mathematics perspective in
\cite{marcone2007}. A similar study can be carried out for interval graphs,
and we summarize here the main results: full details and proofs are included
in the first author's PhD thesis \cite{MFCthesis}, which includes also
results about the subclass of indifference graphs (corresponding to proper
interval orders studied in \cite{marcone2007}).

As customary in reverse mathematics, the system in parenthesis indicates
where the definition is given or the statement proved. Notice also that in
this and in the next section we deal with countable graphs and orders, the
only ones second order arithmetic and its subsystems can speak of.

In the literature it is possible to find slightly different definitions of
interval graphs and orders, which depend on the notion of interval employed.
For example intervals may be required to be closed or not. We thus have five
conceptually distinct definitions of interval graphs:

\begin{definition}[$\rca$]\label{defintgraph}
Let $(V,E)$ be a graph.
\begin{itemize}
  \item $(V,E)$ is an \emph{interval graph} if there exist a linear order
      $(L,<_L)$ and a relation $F\subseteq V \times L$ such that,
      abbreviating $\{x\in L \mid (p,x)\in F\}$ by $F(p)$, for all $p,q\in
      V$ the following hold:
\begin{itemize}
	\item[(i1)] $F(p)\ne \emptyset$ and $\forall x,y \in F(p)\, \forall z\in L\, (x<_L z<_L y \rightarrow z\in F(p))$,
	\item[(i2)] $p \E q \Leftrightarrow  F(p) \cap F(q)\ne\emptyset$.
\end{itemize}
  \item $(V,E)$ is a \emph{1-1 interval graph} if it also satisfies
\begin{itemize}
	\item[(i3)] $F(p)\ne F(q)$ whenever $p\ne q$.
\end{itemize}
  \item $(V,E)$ is a \emph{closed interval graph} if there exist a linear
      order $(L,<_L)$ and two functions $f_L,f_R\colon V\to L$ such that
      for all $p,q\in V$
\begin{itemize}	
    \item[(c1)] $f_L(p)<_L f_R(p)$,
    \item[(c2)] $p \E q \Leftrightarrow  f_L(p)\le_L f_R(q) \le_L f_R(p)
        \lor f_L(q)\le_L f_R(p) \le_L f_R(q) $
\end{itemize}
  \item A closed interval graph $(V,E)$ is a \emph{1-1 closed interval
      graph} if we also have
\begin{itemize}	
    \item[(c3)] $f_R(p) \ne f_R(q) \lor f_L(p) \ne f_L(q)$ whenever $p\ne
        q$.
\end{itemize}
  \item $(V,E)$ is a \emph{distinguishing interval graph} if (c1) and (c2)
      hold together with
\begin{itemize}	
    \item[(c4)] $f_i(p) \ne f_j(q)$ whenever $p\ne q \lor i\ne j$.
\end{itemize}
\end{itemize}
\end{definition}

\subsection{Definitions and characterizations of interval graph}
In \cref{defClosedIntgraph} we mentioned that every interval graph is a
closed interval graph: in fact all the notions introduced in
\cref{defintgraph} are equivalent in a sufficiently strong theory. Our first
results concern the systems where the implications between the notions
introduced in \cref{defintgraph} can be proved. The same investigation for
interval orders was carried out in \cite{marcone2007} and in this respect
interval graphs and interval orders behave similarly. Indeed the proofs of
the results we are going to state either mimic the corresponding proofs for
interval orders or are easily derived from those results.

\cref{defintgraph} enumerates increasingly strong conditions, so that the
implications from a later to an earlier notion are easily proved in $\rca$.
Regarding the other implications we obtain that, as is the case for interval
orders, there are three distinct notions of interval graphs in $\rca$, namely
that of interval, 1-1 interval and closed interval graph.

\begin{theorem}[$\rca$]\label{6}
Every closed interval graph is a distinguishing interval graph.
\end{theorem} 	

\begin{theorem}[$\rca$] \label{StrengthDefInter}
The following are equivalent:
\begin{enumerate}
	\item $\wkl$;
	\item every interval graph is a 1-1 interval graph;
	\item every 1-1 interval graph is a closed interval graph;
	\item every interval graph is a closed interval graph.
\end{enumerate}
\end{theorem} 	

\subsection{Structural characterizations of interval graphs}

Since interval graphs are incomparability graphs (and \cref{DefComparability}
can be given in $\rca$) we first look at the most important structural
characterization of comparability graphs. The first result is due to Jeff
Hirst (\cite[Theorem 3.20]{HirstPhD}).

\begin{lemma}[$\rca$]\label{compgraph}
The following are equivalent:
\begin{enumerate}
\item $\wkl$;
\item every irreflexive graph such that every cycle of odd length has a triangular
    chord is a comparability graph.
\end{enumerate}
\end{lemma}

We then consider two structural characterizations of interval graphs (notice
that \cref{def:triangulated} can be given in $\rca$). The necessity of both
conditions is provable in $\rca$, but the sufficiency of one of them requires
$\wkl$.

\begin{theorem}[$\rca$]	\label{1}
Every interval graph is an incomparability graph such that every simple cycle
of length four has a chord.	Moreover, every interval graph is triangulated
and has no asteroidal triples.

Every incomparability graph such that every simple cycle of length four has a
chord is an interval graph.
\end{theorem} 	

\begin{theorem}[$\rca$]\label{wkltri}
The following are equivalent:
\begin{enumerate}
	\item $\wkl$;
	\item if a reflexive graph is triangulated and has no asteroidal triples, then it is an interval graph.
\end{enumerate}
\end{theorem}

\cref{IntG:figRiassuntiva} summarizes the results about the different
definitions and characterizations of interval graphs. The arrows correspond
to provability in $\rca$, while every implication from a notion below another
is equivalent to $\wkl$.

\begin{figure}
\begin{center}
\begin{tikzpicture}[scale=.3]
%Nodes
\node (dist) at (13,15) {distinguishing interval};
\node (11cl) at (-1,15) {1-1 closed interval};
\node (closed) at (-13,15) {closed interval};
\node (11int) at (-13,10) {1-1 interval};
\node (4) at (2,5) {four cycle + incomparability};
\node (int) at (-13,5) {interval graph};
\node (tri) at (-13,0) {triangulated + no asteroidal triples};
%\node (4chord) at (-13,0) {four-odd cycles with chord};

%Lines
\draw[<->,thick] (closed) -- (11cl);
\draw[<->,thick] (dist) -- (11cl);
\draw[<->,thick] (4) -- (int);
%\draw[<->,thick] (tri) -- (4chord);
\draw[->,thick] (closed) -- (11int);
%\draw[->,thick]  (11int) -- (4);
\draw[->,thick]  (11int) -- (int);
\draw[->,thick]  (int) -- (tri);
%\draw[->,thick]  (int) -- (4chord);
%\draw[->, ultra thick,bluUniud] (-14,-3) .. controls (-19,5) and (-18,7) ..  (-3,17);
%\draw[->, ultra thick,bluUniud] (-9,-2) .. controls (-9.5,0)  ..  (-9,2);
%\draw[->, ultra thick,bluUniud] (-13,-3) .. controls (-18,5) and (-17,8) ..  (-4,11);
%%\draw[->, ultra thick, blue] (16,-4) .. controls (17,-2) and (17,2) ..  (16,4);
\end{tikzpicture}
\caption{Implications in $\rca$}\label{IntG:figRiassuntiva}
\end{center}
\end{figure}
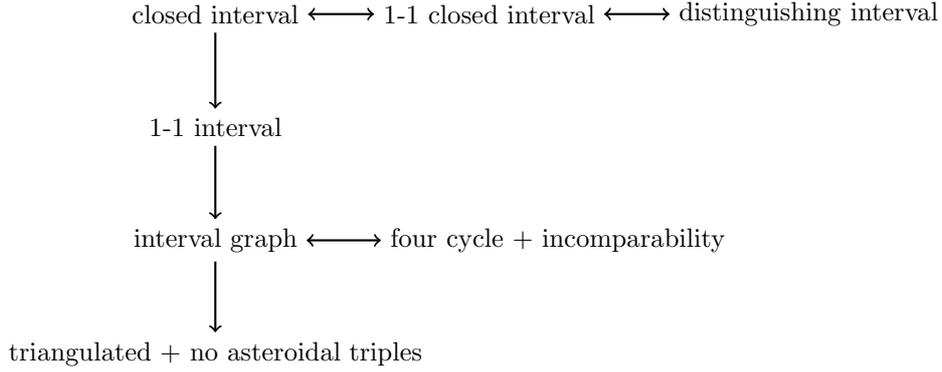

Schmerl \cite{schmerl2005} claimed that the statement \lq\lq A graph is an
interval graph if and only if each finite subgraph is representable by
intervals\rq\rq\ is equivalent to $\wkl$. \cref{wkltri} confirms his claim
and shows that compactness is necessary to prove the statement. On the other
hand, the corresponding statement for interval orders, i.e.\ an order is an
interval order if and only if each suborders is an interval order, is
provable in $\rca$ because the structural characterization of interval orders
(as the partial orders not containing $2 \oplus 2$) is provable in $\rca$
\cite[Theorem 2.1]{marcone2007}. The different strengths of the structural
characterizations of interval graphs and orders can be traced to the fact
that an interval order carries full information about the relative position
of the intervals in its representations, while an interval graph does not.

Lekkerkerker and Boland \cite{lekkeikerker} provide another characterization
of interval graphs listing all the forbidden subgraphs. It is routine to
check in $\rca$ that those graphs are a complete list of graphs whose cycles
of length greater than four do not have chords or which contain an asteroidal
triple.

\subsection{Interval graphs and interval orders}

Different definitions for interval orders, mirroring those of
\cref{defintgraph}, were given and studied in \cite{marcone2007}. We give
here only the most basic one, as the others can be easily guessed from this.

\begin{definition}[$\rca$] \label{defintorder1}
A partial order $(V,\preceq)$ is an \emph{interval order} if there exist a
linear order $(L,<_L)$ and a relation $F\subseteq V \times L$ such that,
abbreviating $\{x\in L \mid (p,x)\in F\}$ by $F(p)$, for all $p,q\in V$ the
following hold:
\begin{itemize}
	\item[(i1)] $F(p)\ne \emptyset$ and $\forall x,y \in F(p)\, \forall z\in L \,(x<_L z<_L y \rightarrow z\in F(p))$,
	\item[(i2)] $p \preceq q \Leftrightarrow \forall x\in F(p)\,\forall y\in F(q)\,(x<_L y)$.
\end{itemize}
\end{definition}

We explore the strength of the statements that allow moving from interval
graphs to interval orders and back. By the previous results (and the
corresponding ones in \cite{marcone2007}) it suffices to consider three
different notions on each side, and we concentrate on the relationship
between corresponding notions. In one direction everything goes through in
$\rca$.

\begin{theorem}[$\rca$]\label{graph-order} %\cite{Fishburn}*{38}}
Let $(V,E)$ be a graph and let $\mathcal{P}$ be any of \lq\lq interval\rq\rq,
\lq\lq 1-1 interval\rq\rq, \lq\lq closed interval\rq\rq. $(V,E)$ is a
$\mathcal{P}$ graph if and only if there exists a $\mathcal{P}$ order
$(V,\prec)$ such that $p \E q \Leftrightarrow p\nprec q \land q\nprec p$ for all $p,q\in V$.
\end{theorem}

The other direction is more interesting, as only in one case $\rca$ suffices.
The proofs of the reversals to $\wkl$ are modifications of the proof of
\cite[Theorem 6.4]{marcone2007}.

\begin{theorem}[$\rca$] \label{order-graph}
Let $(V,\preceq)$ be a partial order. $(V,\preceq)$ is an interval order if
and only if $(V,E)$, where $p \E q \Leftrightarrow p\nprec q \land q\nprec p$
for all $p,q\in V$, is an interval graph.
\end{theorem}

\begin{theorem}[$\rca$] \label{order2}
The following are equivalent:
\begin{enumerate}
\item $\wkl$
\item Let $(V,\preceq)$ be a partial order. $(V,\preceq)$ is a 1-1 interval
    order if and only if $(V,E)$, where $p \E q \Leftrightarrow p\nprec q \land q\nprec p$
    for all $p,q\in V$, is a 1-1 interval graph.
\item Let $(V,\preceq)$ be a partial order. $(V,\preceq)$ is a closed
    interval order if and only if $(V,E)$, where $p \E q \Leftrightarrow
    p\nprec q \land q\nprec p$ for all $p,q\in V$, is a closed interval graph.
\end{enumerate}
\end{theorem}

\section{Why compactness does not suffice} \label{Sec:MainTheoremRM}

It is immediate (using \cref{1}) that \cref{tuniqord,QuniqordRCA} are
provable in $\rca$. On the other hand, we now show that \cref{SufCondition}
is much stronger, and indeed equivalent to $\aca$. As mentioned in the
introduction of the paper, this result explains why the attempts to prove it
by compactness cannot succeed.

\begin{lemma}[$\aca$]\label{ACAproves}
Let $(V,E)$ be a connected reflexive graph which is triangulated and with no asteroidal
triples. Suppose furthermore that $a,b,c,d \in V$ are such that $\neg ab
\Qcon cd$ and $\neg ab \Qcon dc$. Then there exists a buried subgraph $B
\subseteq V$ such that either $a,b\in B$ or $c,d \in B$, and no subgraph $A
\subseteq B$, which contains either $a,b$ or $c,d$ respectively, is a buried
subgraph.
\end{lemma}
\begin{proof}
By \cref{StrengthDefInter,wkltri} $\wkl$, and a fortiori $\aca$, suffices to
prove that any connected graph which is triangulated and with no asteroidal
triples has a closed interval representation. We then need to check that the
proof of \cref{SufCondition}, which indeed provides a buried subgraph with
the desired properties, goes through in $\aca$.

The first step is checking that, given $v,u \in V$ with $\neg v \E u$, we can
carry out \cref{ConstrB} and define $B(v,u)$ and the various $B_n(v,u)$'s in
$\aca$. In fact the definition of each $B_n(v,u)$ in \cref{ConstrB} uses an
instance of arithmetical comprehension and thus the whole construction, as
presented there, appears to require the system known as $\aca^+$, which is
properly stronger than $\aca$.

This problem can however be overcome in the following way. Given $v,u \in V$
as before, we can characterize $B(v,u)$ as the set of $w \in V$ such that
there exists a finite tree $T \subseteq 2^{<\nat}$ and a label function $\ell
\colon T \to V$ with the following properties:
\begin{itemize}
\item $\ell(\emptyset)=w$ (here $\emptyset$ is the root of $T$);
\item if $\sigma \in T$ is not a leaf of $T$, then $\sigma \smf 0,
    \sigma\smf 1 \in T$, $\ell(\sigma) \E \ell(\sigma \smf 0)$ and $\neg
    \ell(\sigma) \E \ell(\sigma \smf 1)$;
\item if $\sigma \in T$ is a leaf of $T$, then $\ell(\sigma) \in \{v,u\}$.
\end{itemize}
In fact, the tree and its label function describe the \lq steps\rq\ allowing
$w$ to enter $B(v,u)$. Moreover $B_n(v,u)$ is the set of $w \in V$ such that
there exists $T \subseteq 2^{<n}$ and $\ell$ witnessing $w \in B(v,u)$. These
characterizations of $B(v,u)$ and $B_n(v,u)$ use $\Sigma^0_1$-formulas, and
show that $\aca$ suffices to prove the existence of the sets.

Once $B(v,u)$ and each $B_n(v,u)$ are defined, it is straightforward to check
that all subsequent steps in the proof of \cref{SufCondition} can be carried
out in $\rca$.
\end{proof}

To prove that \cref{SufCondition} implies $\aca$ we use the following
notions. Given an injective function $f \colon \nat \to \nat$ we say that $i$
is true for $f$ when $f(k)>f(i)$ for all $k>i$. It is easy to see that there
exist infinitely many $i$ which are true for $f$. If $i$ is not true for
$f$, i.e.\ if $f(k)<f(i)$ for some $k>i$, we say that $i$ is false for $f$.
Moreover, we say that $i$ is true for $f$ at stage $s$ if $f(k)>f(i)$ whenever
$i<k<s$, and that $i$ is false for $f$ at stage $s$ if $f(k)<f(i)$ for some
$k$ with $i<k<s$. If the injective function $f$ is fixed, we omit \lq\lq for
$f$\rq\rq\ from this terminology.

The following Proposition is well-known (see e.g.\ the discussion after
Definition 4.1 in \cite{Noeth}).

\begin{proposition}[$\rca$]\label{true}
The following are equivalent:
\begin{enumerate}
\item $\aca$;
\item if $f \colon \nat \to \nat$ is an injective function there exists an
    infinite set $T$ such that every $i \in T$ is true for $f$.
\end{enumerate}
\end{proposition}

\begin{theorem}[$\rca$]
The following are equivalent:
\begin{enumerate}
\item $\aca$;
\item let $(V,E)$ be a connected graph, triangulated and with no asteroidal
    triples; if $a,b,c,d \in V$ are such that $\neg ab \Qcon cd$ and $\neg
    ab \Qcon dc$, then there exists a buried subgraph $B \subseteq V$ such
    that either $a,b\in B$ or $c,d \in B$, and no subgraph $A \subseteq B$,
    which contains either $a,b$ or $c,d$ respectively, is a buried
    subgraph;
\item let $(V,E)$ be a connected closed interval graph; if $(W,Q)$ has more
    than two components, then there exists a buried subgraph $B
    \subseteq V$;
\item let $(V,E)$ be a connected closed interval graph; if $(V,E)$ is not
    uniquely orderable, then there exists a buried subgraph $B \subseteq
    V$.
\end{enumerate}
\end{theorem}
\begin{proof}
$(1 \Imp 2)$ is \cref{ACAproves}. The implication $(2 \Imp 3)$ is trivial,
while $(3 \Imp 4)$ follows directly from \cref{QuniqordRCA}, which goes
through in $\rca$.

To prove $(4 \Imp 1)$ we fix an injective function $f \colon \nat \to \nat$
and we define (within $\rca$) a connected closed interval graph $(V,E)$ such
that $(W,Q)$ has more than two components. We then prove, arguing
in $\rca$, that the unique buried subgraph $B \subseteq V$ codes the
(necessarily infinite) set of numbers which are true for $f$.

We let $V = \{a,b,k,r\} \cup \{x_i, y_i \mid i \in \nat\}$. Beside making sure that $(V,E)$ is reflexive, the definition of the edge
relation is by stages: at stage $s$ we define $E$ on $V_s = \{a,b,k,r\} \cup
\{x_i, y_i \mid i < s\} \subseteq V$. At stage $0$ let $k$ be adjacent to
$a$, $b$ and $r$ (and add no other edges). At stage $s+1$ we define the vertices adjacent to
$x_s$ and $y_s$ by the following clauses:
\begin{enumerate}[label=(\alph*)]
  \item $a \E x_s \E b \E y_s$ and $x_s \E k \E y_s$,
  \item $x_s \E x_i$ and $y_s \E y_i$ for each $i < s$,
  \item $x_s \E y_i$ for each $i \le s$,
  \item for $i \le s$, $y_s \E x_i$ if and only if $i$ is true for $f$ at
      stage $s+1$.
\end{enumerate}

It is immediate that $(V,E)$ is connected. To check that it is a closed
interval graph we define a closed interval representation $f_L,f_R: V \to L$
where $(L,<_L)$ is a dense linear order. The definition of $f_L$ and $f_R$
reflects the construction of the graph by stages. At stage $0$ assign to the
members of $V_0$ elements of $L$ satisfying
\[
f_L(r) <_L f_L(k) <_L f_R(r) <_L f_L(a) <_L f_R(a) <_L f_L(b) <_L f_R(b) <_L f_R(k).
\]
%\begin{multline*}
%f_L(r) <_L f_L(s) <_L f_R(r) <_L f_L(k) <_L f_R(s) <_L f_L(a) = f_L(x_0) <_L f_R(a) <_L \\
%<_L f_L(b) <_L f_L(y_0) <_L f_R(x_0) <_L f_R(b) = f_R(y_0) <_L f_R(k).
%\end{multline*}
This ensures that we are representing the restriction of the graph to $V_0$.

At stage $s+1$, first let $f_L(x_s)= f_L(a)$ and $f_R(y_s)=f_R(b)$ (since
this is done at every stage, we are respecting conditions (a) and (b)). We
thus still need to define $f_R(x_s)$ and $f_L(y_s)$; first of all we make
sure that $f_L(b) <_L f_L(y_i) <_L f_L(y_s) <_L f_R(x_s) <_L f_R(b)$ for
every $i<s$, so that (c) is also respected. To respect condition (d) as well
we satisfy the following requirements:
\begin{itemize}
  \item if $i<s$ is true at stage $s+1$, then $f_R(x_s) <_L f_R(x_i)$ (which
      implies $f_L(y_s) <_L f_R(x_i)$);
  \item if $j<s$ is false at stage $s+1$, then $f_R(x_j) <_L f_L(y_s)$.
\end{itemize}
The existence of $f_L(y_s) <_L f_R(x_s)$ with these properties follows from
the density of $L$ and from the fact that if $i<s$ is true at stage $s+1$ and
$j<s$ is false at stage $s+1$, then $f_R(x_j) <_L f_R(x_i)$. To see this
notice that:
\begin{itemize}
  \item if $i<j$, then $i$ was also true at stage $j+1$ and we set $f_R(x_j)
      <_L f_R(x_i)$ then;
  \item if $j<i$, then $j$ was already false at stage $i+1$ (if $j$ was true
      at stage $i+1$, then $f(j)<f(i)$, and $i$ would be false at stage
      $s+1$ because $j$ is false at that stage), and hence we set $f_R(x_j)
      <_L f_L(y_i) <_L f_R(x_i)$ at that stage.
\end{itemize}
\cref{repr} depicts a sample interval representation following this
construction.

\begin{figure}
\begin{center}
\begin{tikzpicture}[scale=.6]
\draw[|-|] (0,8.5) --  node[above=0mm] {\small $k$} (13,8.5);
%\draw[|-|] (-1,8.5) --node[above=0mm] {\small $r$} (-0.3,8.5);
\draw[|-|] (1,7.5) -- node[above=0mm] {\small $a$} (2.5,7.5);
\draw[|-|] (3,7.5) -- node[above=2mm,right=16mm] {\small $b$} (12,7.5);
\draw[|-|] (-0.5,7.5) -- node[above=0mm] {\small $r$} (0.5,7.5);
\draw[|-|] (4,6.5) -- node[above=2mm,right=12mm] {\small $y_0$} (12,6.5);
\draw[|-|] (1,6) -- node[above=2mm,left=14mm] {\small $x_0$} (9,6);
\draw[|-|] (4.5,5.5) -- node[above=2mm,right=11mm] {\small $y_1$} (12,5.5);
\draw[|-|] (1,5) -- node[above=2mm,left=8mm] {\small $x_1$} (7,5);
\draw[|-|] (5,4.5) -- node[above=2mm,right=10mm] {\small $y_2$} (12,4.5);
\draw[|-|] (1,4) -- node[above=2mm,left=5mm] {\small $x_2$} (6,4);
\draw[|-|] (7.5,3.5) -- node[above=2mm,right=2mm] {\small $y_3$} (12,3.5);
\draw[|-|] (1,3) -- node[above=2mm,left=11mm] {\small $x_3$} (8,3);
\end{tikzpicture}
\caption{\label{repr} Interval representation of $V_4$ in case $1$ (and so $2$) becomes false at stage $3$.}
%\label{fig2}
\end{center}
\end{figure}
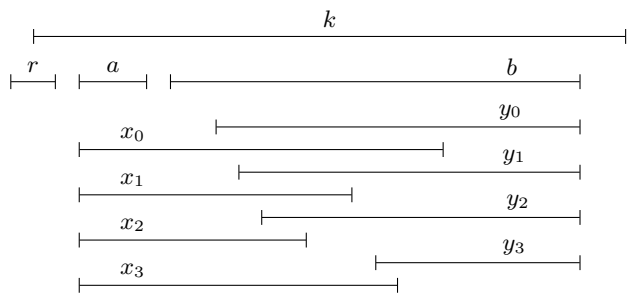

To check that $(V,E)$ is not uniquely orderable let $\prec_1$ be the partial
order induced by the interval representation we just described: $v \prec_1 u$
if and only if $f_R(v) <_L f_L(u)$. Define $\prec_2$ so that $\prec_1$ and
$\prec_2$ coincide on $V \setminus \{s\}$ and $u \prec_2 s$ for all $u \in V
\setminus \{r,k\}$. It is immediate that both $\prec_1$ and $\prec_2$ are
associated to $(V,E)$, and that $\prec_2$ is not dual of $\prec_1$.

%We claim that $\neg ab \Qcon rk$. Notice that, by definition of $E$, $rk \Q ec$ if and only if $e=s$ and $c \in \{a,b,x_n,y_n \mid n \in \nat\}$. Moreover, by definition of $Q$, $sc \Q kd$ holds if and only if $\neg k \E d$ and $c \E d$. It follows from the definition of $E$ that there is no $d$ with the desired properties. We can conclude that $ab$ and $rk$ are not $Q$-connected. Analogous observations show that $\neg ab \Qcon kr$. Hence, let $B$ be a buried subgraph.

By (4) there exists a buried subgraph $B \subseteq V$. First of all notice
that $k \in K(B)$ and hence $k \notin B$. Now observe that $r \in B$ implies,
using Conditions (i) and (iii) of \cref{defbburied}, that either some $x_n$
or some $y_n$ belongs to $B$. From there, using Condition (iii) again, it is
easy to see that $B=V \setminus \{k\}$ and hence $R(B)= \emptyset$,
contradicting Condition (ii). Thus $r \notin B$. Then, in order to satisfy
Condition (i), we must have either $a,b \in B$ or $a,y_n \in B$ or $x_m,y_n
\in B$, for some $n$ and some $m$ which is false at stage $n+1$. In any case
we have $a \in B$: in the first two cases this is obvious, and in the latter
case this follows from Condition (iii) because $a \E x_m$ and $\neg a \E y_n$
for every $n$ and $m$. But then, using $b \E y_n$ and $\neg b \E a$ we obtain
$b \in B$ even in the second and third case. Thus we can conclude that $a,b
\in B$. For each $n$ we have $y_n \E b$ and $\neg y_n \E a$ and therefore
$y_n \in B$. Since $b \E x_n \E a$ for each $n \in \nat$, then either $x_n
\in K(B)$ or $x_n \in B$ depending whether $x_n$ is adjacent to every $y_m$
or not, namely whether $n$ is true or false for $f$. Therefore we showed
\[
B = \{a,b\} \cup \{y_n \mid n \in \nat\} \cup \{x_n \mid n \text{ is false}\},
\]
so that $K(B)= \{k\} \cup \{x_n \mid n \text{ is true}\}$ and $R(B)=\{r\}$.
Then $T=\{n \mid x_n \notin B\}$ is the (necessarily infinite) set of all $n$
which are true for $f$.
\end{proof}

%\bibliographystyle{alpha}
%\bibliography{BibUnicoOrdine}

\begin{thebibliography}{FHM{\etalchar{+}}16}

\bibitem[FC19]{MFCthesis}
Marta Fiori-Carones.
\newblock {\em Filling cages. Reverse mathematics and combinatorial
  principles}.
\newblock PhD thesis, Universit{\`a} di Udine, Italy, 2019.

\bibitem[FHM{\etalchar{+}}16]{Noeth}
Emanuele Frittaion, Matthew Hendtlass, Alberto Marcone, Paul Shafer, and Jeroen
  Van~der Meeren.
\newblock Reverse mathematics, well-quasi-orders, and {N}oetherian spaces.
\newblock {\em Archive for Mathematical Logic}, 55(3-4):431--459, 2016.

\bibitem[Fis70]{fishburn1970}
Peter~C Fishburn.
\newblock Intransitive indifference with unequal indifference intervals.
\newblock {\em Journal of Mathematical Psychology}, 7(1):144--149, 1970.

\bibitem[Fis85]{Fishburn}
Peter~C. Fishburn.
\newblock {\em Interval Orders and Interval Graphs}.
\newblock Wiley, 1985.

\bibitem[Han82]{Hanlon82}
Philip Hanlon.
\newblock Counting interval graphs.
\newblock {\em Transactions of the American Mathematical Society},
  272(2):383--426, 1982.

\bibitem[Hir87]{HirstPhD}
Jeffry~L. Hirst.
\newblock {\em Combinatorics in Subsystems of Second Order Arithmetic}.
\newblock PhD thesis, The Pennsylvania State University, 1987.

\bibitem[Hir15]{Hirschfeldt15}
Denis~R. Hirschfeldt.
\newblock {\em Slicing the {T}ruth}.
\newblock World Scientific, 2015.

\bibitem[LB62]{lekkeikerker}
Cornelis~J. Lekkerkerker and Johan~C. Boland.
\newblock Representation of a finite graph by a set of intervals on the real
  line.
\newblock {\em Fundamenta {M}athematicae}, 51:45--64, 1962.

\bibitem[Mar07]{marcone2007}
Alberto Marcone.
\newblock Interval orders and reverse mathematics.
\newblock {\em Notre {D}ame {J}ournal of {F}ormal {L}ogic}, 48:425--448, 2007.

\bibitem[Sch05]{schmerl2005}
James~H. Schmerl.
\newblock Reverse mathematics and graph coloring: eliminating diagonalization.
\newblock In S.~Simpson, editor, {\em Reverse Mathematics 2001}, volume~21 of
  {\em Lecture Notes in Logic}, pages 331--348. Association of Symbolic Logic,
  La Jolla, CA, 2005.

\bibitem[Sim09]{sosoa}
Stephen~G. Simpson.
\newblock {\em Subsystems of {S}econd Order Arithmetic}.
\newblock Cambridge University Press, second edition, 2009.

\bibitem[Tro97]{TrotterSurvey}
William~T. Trotter.
\newblock New perspectives on interval orders and interval graphs.
\newblock In {\em Surveys in combinatorics, 1997 ({L}ondon)}, volume 241 of
  {\em London Math. Soc. Lecture Note Ser.}, pages 237--286. Cambridge Univ.
  Press, Cambridge, 1997.

\bibitem[Wie14]{wiener1914contribution}
Norbert Wiener.
\newblock A contribution to the theory of relative position.
\newblock {\em Proc. {C}amb. {P}hilos. {S}oc.}, 17:441--449, 1914.

\end{thebibliography}

\newcommand{\etalchar}[1]{$^{#1}$}

\end{document}